\documentclass[letterpaper,leqno]{amsart}
\usepackage{latexsym,amssymb,amsmath,amsthm,color}
\usepackage[english]{babel}
\usepackage{amsfonts}

\numberwithin{equation}{section}
\newcommand\delc[1]{}

\newcommand{\eps}{\varepsilon}
\newcommand{\rH}{\mathrm{ H}}
\newcommand{\lb}{\langle}
\newcommand{\rb}{\rangle}

\newtheorem{theorem}{Theorem}[section]
\newtheorem{lemma}[theorem]{Lemma}
\newtheorem{proposition}[theorem]{Proposition}

\newtheorem{definition}[theorem]{Definition}
\newtheorem{remark}[theorem]{Remark}

\newcommand{\sol}{\mathrm{sol}}

\begin{document}
\title[Stochastic NSEs with irregular noise]{A note on stochastic Navier-Stokes equations with not regular
multiplicative noise}
\author[Z. Brze\'zniak and B. Ferrario]{Zdzis{\l}aw Brze\'zniak \\Department of Mathematics\\
                         The University of York\\
                         Heslington, York YO10 5DD, UK
       \and  Benedetta Ferrario \\Dip. di Matematica ``F. Casorati''\\
                           Universit\`a di Pavia\\
                           I-27100 Pavia}
\dedicatory{Dedicated to Professor Boles{\l}aw Szafirski
on his 80th birthday.}
\date{\today}

\begin{abstract}
We consider the Navier-Stokes equations in $\mathbb R^d$ ($d=2,3$) with 
 a stochastic forcing term which is white noise in time and coloured in space; 
 the spatial covariance of the noise is not too regular, so It\^o calculus cannot 
 be applied in the space of finite energy vector fields. 
We prove existence of weak solutions for $d=2,3$ and pathwise uniqueness for $d=2$.
\end{abstract}



\maketitle

\section{Introduction}
The aim of this paper is to study the stochastic Navier-Stokes equations
with multiplicative noise, that is the equations of motion of a viscous
incompressible fluid with two forcing terms, one is deterministic and the
other one is random. The equations are
\begin{equation}\label{sist:ini}
\begin{cases}
\partial_t v+[-\nu \Delta v +(v \cdot \nabla)v +\nabla p]\ dt = G(v)dw+f\ dt
\\
\nabla \cdot v=0
\end{cases}
\end{equation}
where the unknowns are the velocity $v=v(t,x)$ and  the pressure $p=p(t,x)$.
By $\nu>0$ we denote  the viscosity coefficient  and in our model the stochastic force can depend on the velocity itself.
We consider $x \in \mathbb R^d$ for $d=2,3$ and $t\ge 0$.
The above equations are  associated with an initial condition
\begin{equation}\label{sist:ini0}
v(0,x)=v_0(x)
\end{equation}
where $v_0$ is a divergence free square integrable vector field
on $\mathbb R^d$.

The problem of existence and uniqueness of solutions  on
any time interval has been already studied in the case of a
unbounded domain in \cite{BM2013}, \cite{Brz+Li_2006}
  and in the case of full euclidean space $\mathbb{R}^d$
in \cite{vf}, \cite{MR2004}, \cite{MR2005}, always
with smooth assumptions on the noise term. The problem in bounded
domains has been considered in many papers,
with different assumptions on the noise. From the technical point of view,
it is easier to work in a bounded domain, since the issue of compactness
needed in the proof of existence is more involved in unbounded spatial domains.

Inspired by \cite{FG} (Section 3.4),
we consider the Navier-Stokes equations \eqref{sist:ini}
with a multiplicative noise
whose covariance is not regular enough to allow to use It\^o formula
 in the space  of finite energy velocity vectors, which is the
basic space in which one looks for existence of solutions.
In this way the noise is rougher than in \cite{BM2013},
\cite{vf}, \cite{MR2004}, \cite{MR2005}.
Our original aim was to investigate the existence of  invariant measures
 and stationary solutions for these stochastic Navier-Stokes equations,
following on one side the work of \cite{FG} in bounded domains
and on the other side the method by the first
named authour with Motyl and Ondrej\'at \cite{BM0_2015} for unbounded domains
but with  more regular noise, which is based on \cite{MS1999}.
While working on this problem we realised that the existence or uniqueness
of solutions was left open in some cases;
indeed, \cite{FG} proved existence of martingale and stationary
martingale solutions with rough multiplicative noise only for $d=2$.

The aim of this note is to prove existence of a martingale solution for
the Navier-Stokes equations \eqref{sist:ini} in $\mathbb R^d$ for $d=2$
as well as $d=3$,
when the covariance of the noise is not a
trace class operator in the basic space $H$ of finite energy.
When $d=2$ we prove pathwise uniqueness;
this implies existence of a strong solution too.
We point out that working in Banach spaces
(and using $\gamma$-radonifying operators instead of Hilbert-Schimdt operators)
we can make  assumptions on the covariance of the noise, which are
weaker that in \cite{FG} to get existence of  martingale solutions.
However,
the argument to prove existence of martingale solutions comes from
Section 3.4 of \cite{FG}.
The original problem on invariant measures in unbounded domains
will be studied in a subsequent paper.

As far as the content of the paper are concerned, in Section \ref{sec-mf}
we define the abstract setting in order to write \eqref{sist:ini} as an It\^o
equation in some Banach space. Then Section \ref{sec-e} deals with the
existence result, whereas uniqueness when $d=2$ is proved
in Section \ref{sec-u}.
Definition and properties of $\gamma$-radonifying operators and
some compactness lemmas are given in the Appendix.

\section{Mathematical framework}\label{sec-mf}
We first introduce the functional spaces.

For $1 \le p<\infty$ let $L^p=[L^p(\mathbb R^d)]^d$ with norm
\[
\|v\|_{L^p}=\left(\sum_{k=1}^d \|v^k\|_{L^p(\mathbb R^d)}^p\right)^{\frac1p}
\]
where $v=(v^1,\ldots,v^d)$.\\
Set $J^s=(I-\Delta)^{\frac s2}$.
We define the generalized Sobolev spaces of divergence free vector
distributions as
\begin{eqnarray}
H^{s,p}&=&\{u \in[{\mathcal S}^\prime(\mathbb R^d)]^d:
\|J^s u\|_{L^p}<\infty\},
\\
H^{s,p}_\sol&=&\{u \in H^{s,p} : \nabla \cdot u =0 \}
\end{eqnarray}
for $s \in \mathbb R$ and $1\le p \le \infty$. The divergence
has to be understood in the weak sense.
We have, see \cite{BL},  that
$J^\sigma$ is an isomorphism between $H^{s,p}$ and $H^{s-\sigma,p}$.
Moreover $H^{s_2,p} \subset H^{s_1,p}$ when $s_1<s_2$ and
the dual space of $H^{s,p}$ is the space $H^{-s}_q$ with
$1< q \le \infty$: $\frac 1p+\frac 1q=1$. we
denote by $\langle\cdot, \cdot \rangle$ the $H^{s,p}-H^{-s,q}$ duality bracket:
\[
\langle u,v\rangle =\sum_{k=1}^d \int_{\mathbb R^d} (J^s u^k)(x) \
 (J^{-s}v^k)(x) dx .
\]

In particular, for the Hilbert case $p=2$
we set $\rH=H^{0,2}_\sol$ and, for $s\neq 0$, $\rH^s=H^{s,2}_\sol$;
that is
\[
 \rH=\{v \in[L^2(\mathbb R^d)]^d: \nabla \cdot v =0\}
\]
with scalar product inherited from $[L^2(\mathbb R^d)]^d$.

We recall the Sobolev embedding theorem, see, e.g., \cite[Th. 6.5.1]{BL}.
If  $1< q< p<\infty$ with
\[ \frac 1p=\frac 1q-\frac{r-s}d\]
then  the following inclusion holds
\[
H^{r,q}\subset H^{s,p}
\]
and there exists a constant $C$ (depending  on $r-s,p,q,d$) such that
\[
\|v\|_{H^{s,p}}\le C \|v\|_{H^{r,q}} \;\;\mbox{ for all }
v \in [{\mathcal S}^\prime(\mathbb R^d)]^d.
\]

By Lemma 2.5 of \cite{hw}, see also  Lemma C.1 in \cite{BM2013},
there exists a separable Hilbert space $U$ such that $U$ is a dense subset
of $H^1$ and is compactly embedded in $H^1$.
We also have that
\[
U \subset H^1 \subset H \simeq H^\prime\subset H^{-1}\subset U^\prime
\]
with dense and continuous embeddings, but in addition $H^{-1}$
is compactly embeddded in $U^\prime$.

Now we define the operators appearing in the abstract formulation of
\eqref{sist:ini}. We refer to \cite{kp86} and \cite{T} for the details.

Let $A=-\Pi \Delta$, where $\Pi$ is the projector onto the space of divergence free vector fields. Then
$A$ is a
linear unbounded  operator in $H^{s,p}$ as well as in
$H^{s,p}_\sol$
($s \in \mathbb R$, $1\le p<\infty$), which generates
a contractive and analytic $C_0$-semigroup $\{e^{-tA}\}_{t\ge 0}$.
Moreover, for $t>0$ the operator $e^{-tA}$ is bounded
from $H^{s,p}_\sol$ into $H^{s^\prime,p}_\sol$
with $s^\prime>s$ and there exists a constant $M$ (depending on $s^\prime-s$ and $p$)
such that, see  Lemma 1.2 in  the
Kato-Ponce paper \cite{kp86},
\begin{equation}
\|e^{-tA} v \|_{H^{s^\prime,p}_\sol}\le M (1+t^{-(s^\prime-s)/2}) \|v\|_{H^{s,p}_\sol}
\end{equation}
i.e.
\begin{equation}\label{semigruppo}
\|e^{-tA}  \|_{\mathcal L(H^{s,p}_\sol;H^{s^\prime,p}_\sol)}\le M (1+t^{-(s^\prime-s)/2}) .
\end{equation}

We set
\[
\|\nabla v \|_{L^2}^2=\sum_{k=1}^d \|\nabla v^k\|^2_{L^2}, \;\;\; v\in H^1.
\]
We have $A: H^1 \to  H^{-1}$ and
\[
\langle Av,v\rangle=\|\nabla v\|_{L^2}^2, \;\;\; v\in H^1.
\]
Moreover
\begin{equation}\label{quadrati}
\|v\|_{H^1}^2=\|v\|^2_{L^2}+\|\nabla v\|_{L^2}^2
\end{equation}

We define the bilinear operator $B:H^1\times H^1\to H^{-1}$ as
\[
\langle B(u,v),z\rangle=\int_{\mathbb R^d} (u(x)\cdot \nabla )v(x) \ \cdot z(x) \ dx.
\]
The form $B$  is bounded, since by H\"older and Sobolev inequalities
\begin{equation}\label{stima4-4}
|\langle B(u,v),z\rangle|   \le \|u\|_{L^4} \|\nabla v \|_{L^2} \|v\|_{L^4}\
 \le C \|u\|_{H^1}\|v\|_{H^1} \|z\|_{H^1}.
\end{equation}
Moreover, see e.g. \cite{T},
\begin{equation}\label{scambio}
\langle B(u,v),z\rangle =-\langle B(u,z),v\rangle , \qquad
\langle B(u,v),v\rangle =0 .
\end{equation}
Using \eqref{stima4-4}-\eqref{scambio}
and the fact that $H^1$ is dense in $L^4$, $B$ can be extended to a bounded bilinear operator from
$L^4\times L^4$ to $H^{-1}$ and
\begin{equation}\label{bL4}
\|B(u,v)\|_{H^{-1}}\le \|u\|_{L^4}\|v\|_{L^4}.
\end{equation}
We shall need an estimate of $B(u,v)$ in bigger spaces.

\begin{lemma}\label{stimaB-1-gamma}
Let  $d=2$ and $g \in (0,1)$. Then there exists  a constant $C_g$  such that for all $\;\;\; u \in H^{1-g}$, $v\in H^{\frac {1-g}2}$,
\begin{equation}\label{ineq-B_g}
\|B(u,v)\|_{H^{-1-g}}
+  \|B(v,u)\|_{H^{-1-g}}\le
C_g\|u\|^{\frac {1-g}2}_{H^{-g}} \|u\|^{\frac {1+g}2}_{H^{1-g}}   \|v\|_{H^{\frac {1-g}2}}.
\end{equation}
\end{lemma}
\begin{proof} Since the space $H^{-1-g}$ is dual to $H^{1+g}$ we have

\[
\|B(u,v)\|_{H^{-1-g}}=\sup_{\|\phi\|_{H^{1+g}}\le 1} | \langle B(u,v),\phi\rangle |.
\]
For smooth vectors of compact support, we have
\begin{equation}\label{stim-tril-g}\begin{split}
\langle B(u,v),\phi\rangle&
 =\int_{\mathbb R^2} ([u(x)\cdot \nabla] v(x)) \cdot \phi(x) \ dx
\\
&=\sum_{i,j=1}^2 \int_{\mathbb R^2} \partial_i (u^i(x) v^j(x))\phi^j(x) \ dx\\
&
=-\sum_{i,j=1}^2 \int_{\mathbb R^2} u^i(x) v^j(x) \partial_i \phi^j(x) \ dx
\end{split}\end{equation}
and similarly
\begin{equation}\label{stim-tril-g2}
\langle B(v,u),\phi\rangle
=-\sum_{i,j=1}^2 \int_{\mathbb R^2} v^i(x) u^j(x) \partial_i \phi^j(x) \ dx .
\end{equation}
This holds also for less regular vectors, by density.
For scalar functions, H\"older inequality and Sobolev embedding (in two dimensions) give
\begin{equation}\label{stim-scalari}
\begin{split}
\left|\int_{\mathbb R^2} fgh \ dx \right|
&\le \|f\|_{L_{p_1}} \|g\|_{L_{p_2}} \|h\|_{L_{p_3}}
 \\
&\le C\|f\|_{H^{\frac{1-g}2}} \|g\|_{H^{\frac{1-g}2}}\|h\|_{H^{g}}
\end{split}
\end{equation}
choosing $p_1=p_2=\frac 4{1+g}$ and $p_3=\frac 2{1-g}$ so that
$\frac 1{p_1}+\frac 1{p_2}+\frac 1{p_3} = 1$.

Hence, from \eqref{stim-tril-g}-\eqref{stim-scalari} we get
\[
\|B(u,v)\|_{H^{-1-g}} + \|B(v,u)\|_{H^{-1-g}}
\le
C \|u\|_{H^{\frac{1-g}2}} \|v\|_{H^{\frac{1-g}2}}.
\]
Moreover, by the complex interpolation
$H^{\frac{1-g}2}=\left[ H^{-g},H^{1-g}\right]_{\frac {1+g}2}$
we have
\[
 \|u\|_{H^{\frac{1-g}2}}\le C \|u\|_{H^{-g}}^{\frac{1-g}2} \|u\|_{H^{1-g}}^{\frac{1+g}2}.
\]
This allows to conclude an estimate for  $\|B(u,v)\|_{H^{-1-g}}$; in the same
way we deal with $\|B(v,u)\|_{H^{-1-g}}$.
\end{proof}
\medskip

We notice that in the proof as well in the sequel, we denote by $C$
different generic constants; if we need to specify it,
we label it in a peculiar way.

Finally, we define the noise forcing term.
Given a real separable Hilbert space $Y$
we consider a $Y$-cylindrical Wiener process $w$ defined on a
stochastic basis $(\Omega,\mathbb F,\mathbb P)$
where $\mathbb F=\{\mathbb F_t\}_{t\ge 0}$ is a right continuous filtration, see e.g.
\cite{dpz}. This means that
\[
w(t)=\sum_{j=1}^\infty \beta_j(t)e_j, \;\; t\ge 0,
\]
where $\{e_j\}_{j\in \mathbb N}$ is a complete orthonormal system in $Y$ and
$\{\beta_j\}_{j \in \mathbb N}$
is a sequence of indipendent identically distributed standard Wiener processes defined
on   $(\Omega,\mathbb F,\mathbb P)$.

In the Appendix we recall the definition and basic properties of
$\gamma$-radonifying operators.
For the covariance of the noise we make the following assumptions:
\begin{enumerate}
\item[{\bf (G1)}]
$\exists \  g \in (0,1)$  such that the mapping
$G:H\to \gamma(Y;H^{- g })$
          is well defined, continuous and
\[
\sup_{v \in H} \|G(v)\|_{\gamma(Y;H^{-g})}=:K_{g,2}<\infty
\]
\item[{\bf (G2)}]
$\exists\  g \in (0,1)$  such that the mapping  $G:H\to \gamma(Y;H^{- g ,4}_\sol)$
            is  well defined, continuous and
\[
\sup_{v \in H} \|G(v)\|_{\gamma(Y;H^{-g,4}_\sol)}=:K_{g,4}<\infty
\]
\item[{\bf (G3)}] If assumption {\bf (G1)} holds, then $G$ extends to a Lipschitz continuous map $G:H^{-g}\to \gamma(Y;H^{-g})$, i.e.
\[
\exists\ L_g>0: \; \|G(v_1)-G(v_2)\|_{\gamma(Y;H^{-g})}
    \le L_g\|v_1-v_2\|_{H^{-g}} \qquad\forall v_1, v_2 \in H^{-g}.
\]
\end{enumerate}

\begin{remark}\label{oss-su-G}
i)
A map  $G:H\to \gamma(Y;H^{-g,p}_\sol)$ is well defined iff the map
$J^{-g}G:H\to \gamma(Y;H^{0,p}_\sol)$ is well defined.
Moreover
\[
\|J^{-g}G(v)\|_{\gamma(Y;H^{0,p}_\sol)}=
\|G(v)\|_{\gamma(Y;H^{-g,p}_\sol)}
\le K_{g,p}, \;\; v \in H.
\]
For $g\le 0$, assumption {\bf (G1)} would mean that $G(v):Y\to H$
is a Hilbert-Schmidt operator, but we make weaker assumptions on $G$,
that is we consider $g>0$.
\\
Notice that if any of the three  conditions above
 holds for some $g$, then it holds for any $g'>g$ too.
\\ ii)
In the particular case of $G(v)e_j=\sigma_j(v) e_j$ (with $\sigma_j:
H \to \mathbb R$ and $Y=H$) we have
\[\begin{split}
\|G(v)\|_{\gamma(Y;H^{- g,p}_\sol)}^p
&=
 \|J^{-g} G(v)\|_{\gamma(Y;H^{0,p}_\sol)}^p\\
 &=
\int_{\mathbb R^d} \left(\sum_{j=1}^\infty \sigma_j(v(x))^2
|(J^{-g} e_j)(x)|^2\right)^{\frac p2} dx\\
&\le
\int_{\mathbb R^d} \left(\sum_j \Vert \sigma_j\Vert_{\infty}^2
|(J^{-g} e_j)(x)|^2\right)^{\frac p2} dx,
\end{split}\]
where $\displaystyle\Vert \sigma_j\Vert_{\infty}^2:=\sup_{v \in H} \sigma_j(v)^2$.

Taking $p=2$ we infer that  the condition {\bf (G1)} holds if
\[
\sum_{j=1}^\infty \Vert \sigma_j\Vert_{\infty}^2 \|e_j\|_{H^{-g}}^2<\infty.
\]
However,  condition {\bf (G1)}  is more involved.
\\ iii)
According to Proposition \ref{conv-stoc} we have
\begin{equation}\label{E-conv-stoc}
\mathbb E \left\|\int_0^t G(v(s))\,dw(s)\right\|^m_{H^{-g}}
\le C_m (K_{g,2})^m t^{m/2}.
\end{equation}
\end{remark}

Projecting the first equation of \eqref{sist:ini} onto the space of divergence free vector fields , we get rid of the pressure term
and we can write the stochastic  Navier-Stokes equations \eqref{sist:ini}
in abstract form as
\begin{equation}\label{sns}
\begin{cases}dv(t)+[Av(t)+B\left(v(t),v(t)\right)]dt=G(v(t))\,dw(t)+f(t)\ dt,& t \in (0,T] \\
v(0)=v_0
\end{cases}
\end{equation}
We assume from now on that
$v_0 \in H$ and $f \in L^p(0,T;H^{-1})$ for some $p>2$, see Assumption \textbf{A.2} and Theorem 5.1  in \cite{BM2013}.
Now, we denote by $C([0,T];H_{\mathrm{w}})$ the space of $H$-valued  weakly continuous
functions with the topology of uniform weak convergence on $[0,T]$;
in particular $v_n \to v$ in $C([0,T];H_{\mathrm{w}})$ means
\[
\lim_{n\to \infty} \sup_{0\le t\le T}|(v_n(t)-v(t),h)_H|=0
\]
for all $h \in H$.
Notice that $v(t) \in H$ for any $t$ if $v \in C([0,T];H_{\mathrm{w}})$.

Our aim is to find a martingale solution to \eqref{sns}.
By this we mean a weak solution in the probabilistic sense, according
to the following
\begin{definition}
[solution to the martingale problem]
We say that there exists a martingale solution of \eqref{sns} if there exist
\begin{enumerate}
\item[$\bullet$]
a stochastic basis $(\hat\Omega,\hat{\mathbb F},\hat{\mathbb P})$
\item[$\bullet$]
a $Y$-cylindrical Wiener process $\hat w$
\item[$\bullet$]
a progressively measurable process $v:[0,T]\times \hat \Omega\to H$ with
$\hat{\mathbb P}$-a.e. path
\[
v \in C([0,T];H_{\mathrm{w}})\cap L^2(0,T;L^4)
\]
and for any $t \in [0,T], \psi \in H^2$
\begin{multline}\label{sol-path}
 \langle v(t),\psi\rangle
 +\int_0^t \langle Av(s) ,\psi\rangle ds
 +\int_0^t \langle B(v(s),v(s)),\psi \rangle ds
\\=
 \langle v_0,\psi\rangle
  +\int_0^t \langle f(s) ,\psi\rangle ds
  +\langle \int_0^t G(v(s))\,d\hat w(s),\psi\rangle
\end{multline}
$\hat{\mathbb P}$-a.s.
\end{enumerate}
\end{definition}
The regularity of the paths of this solution makes all the terms in \eqref{sol-path} well defined, thanks to
\eqref{bL4} and \eqref{E-conv-stoc}.

\section{Existence of solutions}\label{sec-e}
Looking for martingale solutions for system \eqref{sns} one cannot use It\^o
calculus in the space $H$, since the covariance of the noise is not regular
 enough.
Therefore, we introduce an approximating system by regularizing the covariance
of the noise; this gives a sequence of approximating processes $\{v_n\}_{n}$.
In order to pass to the limit as $n\to\infty$
we need the tightness of the sequence of their laws.
This is obtained by working with two auxiliary processes $z_n$ and $u_n$
with $v_n=z_n+u_n$ in a similar way to \cite{FG}.

Therefore, we first introduce a smoother problem which approximates
\eqref{sns};
then we prove the tightness of the sequence of the laws; finally
we show the convergence, providing existence of a martingale solution
to \eqref{sns}.

\subsection{The approximating equation}
We start by defining the sequences
\[
R_n=n(nI+A)^{-1} \qquad \qquad   G_n=R_n G, \qquad n=1,2,\ldots
\]
We have, see \cite{P}, Sect. 1.3, that each $R_n$ is a contraction operator in $H^{s,p}_\sol$  and
it converges strongly to the identity operator, i.e.
\[
\|R_n\|_{\mathcal L(H^{s,p}_\sol;H^{s,p}_\sol)} \le 1
\qquad \text{and} \qquad\lim_{n \to \infty} R_n h=h
\qquad \forall h \in H^{s,p}_\sol.
\]
Moreover, it is a linear bounded operator from $H^{s,p}_\sol$ to $H^{s+t,p}_\sol$ for any $t \le 2$; but the operator norm is not uniformly bounded in $n$ for $t>0$. We point out the case for $p=2$, needed in the sequel; we have
\[
\|R_n\|_{\mathcal L(H^s_\sol;H^{s+t}_\sol)} =\sup_{\|u\|_{H^s_\sol}\le 1} \|R_n u\|_{H^{s+t}_\sol}
\]
and denoting by $\hat u=\hat u(\xi)$ the Fourier transform of $u$
\begin{equation}\label{stime-non-unif}
\|R_n u\|_{H^{s+t}_\sol}=\|J^t R_n J^s u\|_{L^2}=\|n\frac{(1+|\xi|^2)^{\frac t2}}{n+|\xi|^2}(1+|\xi|^2)^{\frac s2}\hat u(\xi)\|_{L^2}
\le B_n \|u\|_{H^s_\sol}
\end{equation}
with $B_n:=\|n\frac{(1+|\xi|^2)^{\frac t2}}{n+|\xi|^2}\|_{L_\infty}
=C_t \frac n{(n-1)^{1-\frac t2}}$, given $t \in (0,2)$. For large $n$,
the quantity  $B_n$ behaves like  $n^{\frac t2}$, which is not bounded.

From the above and Baxendale \cite{Bax-76}
\begin{equation}\label{stimaGn}
\|G_n(v)\|_{\gamma(Y;H^{-g,p}_\sol)}\le \|G(v)\|_{\gamma(Y;H^{-g,p}_\sol)}\qquad \forall n
\end{equation}
and
\begin{equation}\label{lim-Gn}
\lim_{n \to \infty} \|G_n(v)-G(v)\|_{\gamma(Y;H^{-g,p}_\sol)}=0.
\end{equation}
On the other hand, the operator $G_n(v)$ is more regular than $G(v)$. Indeed,
assuming {\bf (G1)},   $G_n(v)$ is a Hilbert-Schmidt operator in $H$, i.e.
\begin{equation}\label{laGn}\begin{split}
\|G_n(v)\|_{\gamma(Y;H)}&\le
\|R_nJ^g\|_{\mathcal L(H;H)}  \|J^{-g} G(v)\|_{\gamma(Y;H)}\\
&
\le \|R_n\|_{\mathcal L(H;H^g)}  \|G(v)\|_{\gamma(Y;H^{-g})} .
\end{split}\end{equation}

For any $n \in \mathbb N$, we consider the following auxiliary problem
\begin{equation}\label{snsn}
\begin{cases}
dv(t)+[Av(t)+B(v(t),v(t))]dt=G_n(v(t))\,dw(t)+f(t) dt,  \ t \in (0,T]\\
 v(0)=v_0
\end{cases}\end{equation}
This is the Navier-Stokes equation \eqref{sns} with a more regular noise.
Thanks to \eqref{laGn}, the operator $G_n$ is regular enough and one can prove that
there exists a martingale solution  for each $n$. The result is obtained by means
of It\^o formula for $d\|v_n(t)\|_H^2$, as in Theorem 5.1
of \cite{BM2013} or Theorem 2.1 of \cite{MR2005}.
More precisely
\begin{proposition}\label{esist-vn}
Let $v_0 \in H$ and $f \in L^p(0,T;H^{-1})$ for some $p>2$. If the assumption {\bf (G1)} is satisfied,
then for each $n$  there exists a martingale solution
$\left((\Omega_n,{\mathbb F}_n, {\mathbb P}_n), w_n,v_n\right)$ of \eqref{snsn};
in addition, there exists a constant $C_n$, depending also on $T$, $\|f\|_{L^p(0,T;H^{-1})}$,
$\sup_{v\in H} \|G_n(v)\|_{\gamma(Y;H)}$, such that
\begin{equation}\label{teoGn}
\mathbb E_n \left[
\sup_{0\le t \le T} \|v_n(t)\|_{H}^2+\int_0^T \|\nabla v_n(t)\|^2_{L^2}dt
\right]\le C_n.
\end{equation}
Moreover, if
 $d=2$ then   $v_n \in C([0,T];H)$ $\mathbb P_n$-a.s.
\end{proposition}

\subsection{Tightness}
The  estimates \eqref{teoGn} are not uniform with respect to  $n$, since
\eqref{stime-non-unif} and \eqref{laGn} show that the
Hilbert-Schmidt norms of the $G_n(v)$ are not uniformly bounded.
Therefore from \eqref{teoGn} one cannot get the tightness of the sequence of the laws of $v_n$.
For this reason, in order to find suitable uniform estimates for the sequence $\bigl(u_n\bigr)_{n}$ we will follow a different path, which in fact has been used in \cite{FG} as well as other works, see for instance \cite{Brz+Gat_1999}. To be precise,
we split our problem in two subproblems in the unknowns $z_n$ and $u_n$ with
$v_n=u_n+z_n$. Given the processes $v_n$ and $W_n$ from  Proposition \ref{esist-vn},
we define the process $z_n$ as the solution of the Ornstein-Uhlenbeck equation
\begin{equation}
dz_n(t)+Az_n(t)\ dt=G_n(v_n(t))\,dw_n(t),\;t\in (0,T];\qquad z_n(0)=0.
\end{equation}
Therefore the process $u_n=v_n-z_n$ solves
\begin{equation}
\frac {d u_n}{dt}(t) +Au_n(t) +B(v_n(t),v_n(t))=f(t),\;t\in (0,T];
\qquad u_n(0)=v_0
\end{equation}

We will first analyze the Ornstein-Uhlenbeck processes; each $z_n$  satisfies the following identity
\begin{equation}
z_n(t)=\int_0^t e^{-(t-s)A}G_n(v_n(s))\,dw_n(s) .
\end{equation}

We have two regularity results. The first of them shows the difference between our approach and that of \cite{FG}, as we employ properties of the It\^o stochastic integral in the 2-smooth Banach spaces as $L^4$ and $H^{\varepsilon,4}_\sol$. The second result also depends on this different approach as it is based on ceratin result from \cite{Brz_1997} established 2-smooth Banach spaces.

\begin{lemma}\label{lemma-z2}
Assume conditions {\bf (G1)} and {\bf (G2)}.
Take any $ g _0 \in [ g ,1)$  and put  $\varepsilon= g _0- g \ge 0$.
Then  for any integer $m\ge 2$
there exists a constant $C$ independent of $n$ (but depending on $m$,
$T$ and $g_0$) such that
\[
\mathbb E_n \|z_n\|_{L^m(0,T;H^{\varepsilon,4}_\sol)}^m
\le (K_{g,4}M)^m C.
\]
In particular $z_n \in L^m(0,T;H^{0,4}_\sol)$ $\mathbb P_n$-a.s.
\end{lemma}

\begin{proof}
From Proposition \ref{esist-vn}, for any $s \in [0,T]$ we have
$v_n(s)\in H$.
We will use  Proposition \ref{conv-stoc}
and the fact that {\bf (G2)} and \eqref{stimaGn} imply
that $J^{-g}G_n(v_n(s))\in \gamma(Y;H^{0,4}_\sol)$, for any $s \in [0,T]$. Moreover, since $J^{\varepsilon}=J^{ g _0} J^{\varepsilon- g _0}
=J^{g _0} J^{-g}$, by \cite{Brz_1997} we get
\[\begin{split}
\mathbb E_n \|z_n(t)\|_{H^{\varepsilon,4}_\sol}^m&=
\mathbb E_n \|J^{\varepsilon} z_n(t)\|_{H^{0,4}_\sol}^m\\
&\le  C_m \mathbb E_n\left(\int_0^t
    \| J^{\varepsilon } e^{-(t-s)A}G_n(v_n(s))\|^2_{\gamma(Y;H^{0,4}_\sol)} ds\right)^{m/2}
\\
&\le  C_m \mathbb E_n\left(\int_0^t \|J^{g _0} e^{-(t-s)A}\|_{\mathcal L(H^{0,4}_\sol;H^{0,4}_\sol)}^2
              \|J^{-g} G_n(v_n(s))\|_{\gamma(Y;H^{0,4}_\sol)}^2 ds\right)^{m/2}
\\
&\le  C_m (K_{g,4})^m \left( \int_0^t [M^2+\frac {M^2}{(t-s)^{ g _0}}] ds \right)^{m/2}
 \quad \text{ by } \eqref{semigruppo}\\
&= C_m  (K_{g,4} M )^{m} \left(t+\tfrac 1{1- g _0}t^{1- g _0} \right)^{\frac m2}
    \;\;\text{ since }  g _0<1.
\end{split}\]
Integrating in time we get the result.
\end{proof}

\smallskip
For $0<\beta<1$ let $C^\beta([0,T];H^\delta)$ be  the Banach space
of $H^\delta$-valued $\beta$-H\"older continuous functions endowed with the following  norm
\[
\|z\|_{C^\beta([0,T];H^\delta)}=\sup_{0\le t \le T} \|z(t)\|_{H^\delta}
+
\underset{0\le t< s\le T }{\sup}
\frac{\|z(t)-z(s)\|_{H^\delta}}{|t-s|^\beta}.
\]
\begin{lemma}\label{lemma-z1}
Assume {\bf (G1)} and let
\begin{equation}\label{cond-beta}
0\le\beta<\frac {1-g}2 .
\end{equation}
Then,
for any $p\ge 2$ and $\delta\ge 0$ such that
\begin{equation}\label{cond-zc}
\beta+\frac \delta 2+\frac 1p<\frac {1-g}2
\end{equation}
there exists a modification ${\tilde z}_n$ of $z_n$  such that
\begin{equation}\label{mediaz}
\mathbb E_n\|{\tilde z}_n\|^p_{C^\beta([0,T];H^\delta)}\le \tilde C
\end{equation}
for some constant $\tilde C$  independent of $n$ (but depending
 on $T$, $\beta$, $\delta$ and $p$).
\end{lemma}

\begin{proof}
We can see this result as a
special case of Corollary 3.5 of \cite{Brz_1997}. Let us fix $\beta$, $p\ge 2$ and $\delta\ge 0$ as in conditions \eqref{cond-beta} and \eqref{cond-zc}.  Then
there exists a modification $\tilde z_n$ of $z_n$ such that
\begin{equation}\label{stime-holder}
\mathbb E_n \|\tilde z_n\|^p_{C^\beta([0,T];H^\delta)}\le  C
\mathbb E_n \int_0^T \|G_n(v_n(s))\|^p_{\gamma(Y;H^{-g})}ds.
\end{equation}
Using \eqref{stimaGn} and assumption {\bf (G1)} we conclude the proof of
\eqref{mediaz}. 

\end{proof}

In what follow we will collect some easy consequences of the previous two fundamental results.

Firstly, taking $\delta=0$, $\beta<\frac {1-g}2$ and $p$ big enough in
\eqref{stime-holder}, we get
\begin{equation}
\mathbb E_n\| z_n\|_{C^\beta([0,T];H)}\le
\left(\mathbb E_n\|z_n\|^p_{C^\beta([0,T];H)}\right)^{\frac 1p}\le
\left(C T\right)^{\frac 1p} K_{g,2}.
\end{equation}
Secondly, taking $\delta=\frac {1-g}2$, $\beta=0$ and $p$ big enough in
\eqref{stime-holder}, we get
\begin{equation}
\mathbb E_n\| z_n\|_{C([0,T];H^{\frac {1-g}2})}\le
\left(\mathbb E_n\|z_n\|^p_{C([0,T];H\frac {1-g}2)}\right)^{\frac 1p}\le
\left(C T\right)^{\frac 1p} K_{g,2}.
\end{equation}

Hence, a consequence of the above two lemmas is that there exist  finite constants
$E_m$, $E_g$ and $E_{\beta,\delta}$
such that
\[
\sup_n \mathbb E_n \|z_n\|_{L^m(0,T;H^{0,4}_\sol)}^m=(E_m)^m
\]
\[
\sup_n \mathbb E_n\|z_n\|_{C([0,T];H^{\frac{1-g}2})}=E_g
\]
and
\[
\sup_n \mathbb E_n\|z_n\|^p_{C^\beta([0,T];H^\delta)}=(E_{\beta,\delta})^p
\]
where $\beta$ and $\delta$ are as in Lemma \ref{lemma-z1}.

Therefore, by the Chebyshev inequality we infer that for any given  $\eta>0$
\begin{equation}\label{Chebxz}
\sup_n \mathbb P_n(\|z_n\|_{L^m(0,T;H^{0,4}_\sol)}>\eta)\le \frac {E_m}\eta,
\end{equation}
\begin{equation}\label{Chebxz1}
\sup_n \mathbb P_n(\|z_n\|_{C([0,T];H^{\frac{1-g}2})}>\eta)
\le \frac {E_g}\eta,
\end{equation}
\begin{equation}\label{Chebxz2}
\sup_n \mathbb P_n(\|z_n\|_{C^\beta([0,T];H^\delta)}>\eta)
\le \frac {E_{\beta,\delta}}\eta.
\end{equation}
The last three inequalities allow us  to get uniform estimates in probability for the sequence $u_n$, see Proposition \ref{esiste-u},  and consequently,  for $v_n=z_n+u_n$, see Proposition \ref{stimexvn}.

We recall that each $u_n$ solves a deterministic equation with
two forcing terms: one is  random and the other is  deterministic. Indeed
$u_n$ solves
\begin{equation}\label{eq-un}
\dfrac {d u_n}{d t}(t) +Au_n(t) =-B(v_n(t),v_n(t))+f(t)
\end{equation}
with $u_n(0)=v_0$.
We analyze the above equation \eqref{eq-un}  pathwise.

We have the following fundamental results about the uniform estimates of the approximating processes $(u_n)$.
\begin{proposition}\label{esiste-u}
Assume {\bf (G1)} and {\bf (G2)}.
Let $v_0 \in H$ and $f\in L^p(0,T;H^{-1})$ for some $p>2$.
\\
Then, for any $n$ the paths of the process $u_n=v_n-z_n$ solving
\eqref{eq-un} are such that
\[
 u_n\in L^\infty(0,T;H)\cap L^2(0,T;H^1) \cap L^{\frac 8d}(0,T;L^4) \cap C^{1-\frac d4}([0,T];H^{-1})
\]
$\mathbb P_n$-a.s., and
for any $\varepsilon >0$ there exist positive constants
$C_i=C_i(\varepsilon)$ ($i=6,\ldots,10$) such that
\begin{align}
&\sup_n \mathbb P_n(\|u_n\|_{L^\infty(0,T;H)}>C_6)\le \varepsilon\label{S1}
\\
&\sup_n \mathbb P_n(\|u_n\|_{L^2(0,T;H^1)}>C_7)\le \varepsilon\label{S2}
\\
&\sup_n \mathbb P_n(\|u_n\|_{L^{\frac4{1-g}}(0,T;H^{\frac{1-g}2})}>C_8)\le \varepsilon
\label{S3}\\
&\sup_n \mathbb P_n(\|u_n\|_{L^{\frac 8d}(0,T;L^4)}>C_9)\le \varepsilon\label{S4}
\\
&\sup_n \mathbb P_n(\|u_n\|_{C^{1-\frac d4}([0,T];H^{-1})}>C_{10})\le \varepsilon
\label{S5}
\end{align}
\end{proposition}

\begin{proof}
By definition and merging the regularity of $v_n$ and $z_n$ we have that
$u_n=v_n-z_n \in C([0,T];H_{\mathrm{w}})$   $\mathbb P_n$-a.s.

Using \eqref{bL4} and the Gagliardo-Nirenberg inequality, for $d=2$ or $d=3$,
\begin{equation}\label{GaNi}
\|v\|_{L^4}\le C \|v\|_{L^2}^{1-\frac d4} \|\nabla v\|_{L^2}^{\frac d4}
\end{equation}
 we get
\begin{align}
\int_0^T\|B(v_n(s),v_n(s))\|^{\frac 4d}_{H^{-1}}ds
  &\le \int_0^T \|v_n(s)\|_{L^4}^{\frac 8d} ds \label{B-L4}\\
  &\le C \|v_n\|^{\frac 8d -2}_{L^\infty(0,T;L^2)}
  \int_0^T \|\nabla v_n(s)\|^2_{L^2}  ds \label{B-L2-H1}.
\end{align}
This means that (pathwise) the r.h.s. of equation \eqref{eq-un} belongs to
 $L^{\frac 4d}(0,T;H^{-1})$.

Now, let us  to prove that $u_n \in L^{\frac 4d}(0,T;H^1)$.
For $d=2$, it is a classical result for the Stokes equation, see \cite{T}.
For $d=3$ let us use  classical maximal  regularity theory for linear parabolic
equations, see e.g. \cite{dS}.
We write  $u_n$ as a mild solution
\[
u_n(t)=e^{-tA}v_0+\int_0^te^{-(t-s)A}[f(s)-B(v_n(s),v_n(s))]ds =:e^{-tA}v_0+g_n(t),\; t\in [0,T].
\]
Then,   using  the result from  \cite{dS} we have the following.
If $ f-B(v_n,v_n) \in L^{\frac 43}(0,T;H^{-1})$, then  $g_n \in L^{\frac 43}(0,T;H^{1})$.
Moreover, since  by \eqref{semigruppo},
$\|e^{-tA}v_0\|_{H^1}\le M(1+\frac 1 {t^{1/2}})\|v_0\|_{H}$  for $t>0$,
we infer that for any $p<2$,
$e^{-\cdot A}v_0\in L^p(0,T;H^1)$. Hence, for $d=3$ we get
$u_n \in L^{\frac 43}(0,T;H^1)$.
\\
Thus, the term
 $\langle B(v_n(t),v_n(t)), u_n(t)\rangle $, appearing in a while,  is meaningful for almost every $t$ thanks to \eqref{stima4-4} and is equal to $-\langle B(v_n(t),u_n(t)), v_n(t)\rangle$ thanks to \eqref{scambio}.

Now we look for suitable estimates giving the tightness of the sequence
of laws of $u_n$.
We begin with the usual energy estimate:
\begin{equation}\label{base-ener}
\frac 12 \frac {d}{dt}\|u_n(t)\|_{L^2}^2 +\|\nabla u_n(t)\|_{L^2}^2
= -\langle B(v_n(t),v_n(t)), u_n(t)\rangle + \langle f(t), u_n(t)\rangle
\end{equation}
We consider the trilinear term; by means of \eqref{scambio}-\eqref{stima4-4},
the Gagliardo-Nirenberg inequality \eqref{GaNi} and the Young inequality
we get
\[
\begin{split}
-\langle B(v_n,v_n), u_n\rangle
&=\langle B(v_n,u_n), v_n\rangle\\
&=\langle B(v_n,u_n), z_n\rangle+ \langle B(v_n,u_n), u_n\rangle\\
&=\langle B(v_n,u_n), z_n\rangle\\
&\le \|v_n\|_{L^4} \|\nabla u_n\|_{L^2}\|z_n\|_{L^4}\\
&\le \|u_n\|_{L^4} \|\nabla u_n\|_{L^2}\|z_n\|_{L^4}+
 \|z_n\|_{L^4} \|\nabla u_n\|_{L^2}\|z_n\|_{L^4}\\
&\le C \|u_n\|^{1-\frac d4}_{L^2}\|\nabla u_n\|_{L^2}^{1+\frac d4}\|z_n\|_{L^4}+
  \|\nabla u_n\|_{L^2}\|z_n\|_{L^4}^2\\
&\le \frac 14 \|\nabla u_n\|_{L^2}^2
   +C \|u_n\|_{L^2}^2 \|z_n\|_{L^4}^{\frac 8{4-d}}+  C \|z_n\|_{L^4}^4 .
\end{split}
\]
Estimating
\[
|\langle f, u_n\rangle|\le \|f\|_{H^{-1}}\|u_n\|_{H^1}\le \frac 14 \|u_n\|_{H^1}^2
+C \|f\|_{H^{-1}}^2=\frac 14 \|u_n\|_{L^2}^2+\frac 14 \|\nabla u_n\|_{L^2}^2
+C \|f\|_{H^{-1}}^2
\]
from \eqref{base-ener} we get
\begin{multline*}
\frac 12 \frac {d}{dt}\|u_n(t)\|_{L^2}^2 +\|\nabla u_n(t)\|_{L^2}^2
\le
\frac 12 \|\nabla u_n(t)\|_{L^2}^2\\+ \frac 14 \|u_n(t)\|_{L^2}^2+
C\|z_n(t)\|_{L^4}^{\frac 8{4-d}}
\|u_n(t)\|_{L^2}^2+ C \|z_n(t)\|_{L^4}^4+C \|f(t)\|_{H^{-1}}^2
\end{multline*}
i.e.
\begin{equation}\label{diseq-vn}
\frac {d}{dt}\|u_n(t)\|_{L^2}^2 +\|\nabla u_n(t)\|_{L^2}^2\\
\le \phi_n(t) \|u_n(t)\|_{L^2}^2+\psi_n(t)
\end{equation}
with
\[ \phi_n(t)=\frac 12 +2C\|z_n(t)\|_{L^4}^{\frac 8{4-d}} \mbox{ and }
\psi_n(t)=2C \|z_n(t)\|_{L^4}^4+2C \|f(t)\|_{H^{-1}}^2, \;\; t\in [0,T],\]
By Lemma \ref{lemma-z2}, we have $\phi_n,\psi_n \in L^1(0,T)$  uniformly in $n$.
Hence,  from  Gronwall Lemma applied to inequality
\begin{equation}
\frac {d}{dt}\|u_n(t)\|_{L^2}^2\le \phi_n(t) \|u_n(t)\|_{L^2}^2+\psi_n(t)
\end{equation}
we infer that  there exist constants $C_1,C_2>0$ such that for all $n$
\begin{equation}\label{vH}
\begin{split}
\sup_{0\le t\le T} \|u_n(t)\|_{L^2}^2 &\le
\|v_0\|_H^2 e^{\int_0^T \phi_n(r) dr}+
 \int_0^T e^{\int_s^T \phi_n(r) dr}\psi_n(s)ds\\
&\le \|v_0\|_H^2 e^{\int_0^T \phi_n(r) dr}+
 e^{\int_0^T \phi_n(r) dr}\int_0^T \psi_n(s)ds
 \\
         & \le  C_1 \|v_0\|_H^2 +C_2.
\end{split}
\end{equation}
Next, integrating inequality \eqref{diseq-vn} in time    we get
\begin{equation}\label{vV}
\begin{split}
\int_0^T \|\nabla u_n(t)\|_{L^2}^2 dt&\le \|v_0\|_H^2
       +\int_0^T \Big( \phi_n(t) \|u_n(t)\|_{L^2}^2+\psi_n(t)\Big)dt\\
&\le \|v_0\|_H^2
       +\left(\sup_{0\le t\le T}\|u_n(t)\|_{L^2}^2 \right)
        \int_0^T \phi_n(t) dt +\int_0^T  \psi_n(t)dt
        \\
        & \le  C_3 \|v_0\|_H^2 +C_4.
        \end{split}
\end{equation}

Bearing in mind inequality
\eqref{Chebxz}, we infer  that for any $\varepsilon>0$ there exist  constants
$\eta_1, \eta_2>0$ such that
\[
\sup_n P(\textstyle\int_0^T \phi_n(t) dt>\eta_1)\le \varepsilon, \qquad
\displaystyle
\sup_n \mathbb{P}(\int_0^T \psi_n(t) dt>\eta_2)\le \varepsilon.
\]
Therefore, from \eqref{vH}-\eqref{vV} we get that for any $\varepsilon>0$
there exist suitable constants  $R_1,R_2>0$  such that
\[
\sup_n \mathbb{P}(\|u_n\|_{L^\infty(0,T;H)}>R_1)\le \varepsilon
\qquad
\sup_n \mathbb{P}(\|\nabla u_n\|_{L^2(0,T;L^2)}>R_2)\le \varepsilon .
\]
This proves \eqref{S1}.\\
The last two inequalities  also imply that for any $\varepsilon>0$ there exists a suitable constant  $R_3$ such that
\[
\sup_n \mathbb{P}(\|u_n\|_{L^2(0,T;H^1)}>R_3)\le \varepsilon .
\]
This proves \eqref{S2}.\\
Since  $H^{\frac{1-g}2}=\left[H,H^1\right]_{\frac{1-g}2}$, by the H\"older inequality and the properties of  the complex interpolation
we infer that  $\|u_n\|_{L^{\frac4{1-g}}(0,T;H^{\frac{1-g}2})}
\le C \|u_n\|_{L^\infty(0,T;H)}^{\frac{1+g}{2}} \|u_n\|_{L^2(0,T;H^1)}^{\frac{1-g}2}$.
Therefore
for any $\varepsilon>0$ there exists a  suitable constant  $R_4$ such that
\[
\sup_n \mathbb{P}(\|u_n\|_{L^{\frac4{1-g}}(0,T;H^{\frac{1-g}2})}>R_4)\le \varepsilon .
\]
This proves \eqref{S3}.

By means of the Gagliardo-Nirenberg inequality \eqref{GaNi} we also deduce  that
for any $\varepsilon>0$
there exists a suitable constant $R_5>0$  such that
\[
\sup_n \mathbb{P}(\|u_n\|_{L^{\frac 8d}(0,T;L^4)}>R_5)\le \varepsilon.
\]
This proves \eqref{S4}.

From  equation \eqref{eq-un} for $u_n$  we infer that
\[\begin{split}
\|\frac{du_n}{dt}\|_{L^{\frac 4d}(0,T;H^{-1})} &\le \|Au_n\|_{L^{\frac 4d}(0,T;H^{-1})}+
\|B(v_n,v_n)\|_{L^{\frac 4d}(0,T;H^{-1})}+\|f\|_{L^{\frac 4d}(0,T;H^{-1})}
\\
&\le  \|u_n\|_{L^{\frac 4d}(0,T;H^1)}+\|v_n\|^2_{L^{\frac 8d}(0,T;L^4)}
    + \|f\|_{L^{\frac 4d}(0,T;H^{-1})}  \text{ using } \eqref{B-L4}\\
&\le C \|u_n\|_{L^2(0,T;H^1)}+ 2 \|u_n\|^2_{L^{\frac 8d}(0,T;L^4)}
  +2 \|z_n\|^2_{L^{\frac 8d}(0,T;L^4)}+C \|f\|_{L^2(0,T;H^{-1})}.
\end{split}
\]
Using \eqref{S1}, \eqref{S4} and \eqref{Chebxz},
we find that for any $\varepsilon>0$ there exists a suitable constant   $R_6>0$ such that
\[
\sup_n \mathbb{P}(\|\frac {du_n}{dt}\|_{L^{\frac 4d}(0,T;H^{-1})}>R_6)\le \varepsilon.
\]
By the Sobolev embedding theorem,
$H^{1,\frac4d}(0,T) =\{u \in L^{\frac 4d}(0,T):
u^\prime \in L^{\frac 4d}(0,T)\}\subset C^{1-\frac d4}([0,T])$.
Hence we infer  that there exists a  constant    $R_7>0$ such that
\[
\sup_n \mathbb{P}(\|u_n\|_{C^{1-\frac d4}([0,T];H^{-1})}>R_7)\le \varepsilon.
\]
This proves \eqref{S5}. This completes the proof of Proposition \ref{esiste-u}.
\end{proof}

Now we need to apply a  tightness argument in order to pass to the limit.
Merging the estimates \eqref{Chebxz}-\eqref{Chebxz2} for the process $z_n$
and those for $u_n$ from Proposition \ref{esiste-u} we get  estimates for $v_n=z_n+u_n$.  These estimates in probability are uniform with respect to $n$.
\begin{proposition}\label{stimexvn}
Assume {\bf (G1)} and {\bf (G2)}, let $v_0 \in H$ and $f\in L^p(0,T;H^{-1})$ for some $p>2$
and let $\left((\Omega_n,{\mathbb F}_n, {\mathbb P}_n), w_n,v_n\right)$
be a martingale solution of \eqref{snsn} as given in Proposition \ref{esist-vn}.
\\
Then there exist $\gamma, \delta >0$ such that
for any $\varepsilon >0$ there exist positive constants
$C_i=C_i(\varepsilon)$ ($i=1,\ldots,5$) such that
\[
\sup_n \mathbb P_n(\|v_n\|_{L^\infty(0,T;H)}>C_1)\le \varepsilon
\]
\[
\sup_n \mathbb P_n(\|v_n\|_{L^2(0,T;H^\delta)}>C_2)\le \varepsilon
\]
\[
\sup_n \mathbb P_n(\|v_n\|_{L^{\frac4{1-g}}(0,T;H^{\frac{1-g}2})}>C_3)\le \varepsilon
\]
\[
\sup_n \mathbb P_n(\|v_n\|_{L^{\frac 8d}(0,T;L^4)}>C_4)\le \varepsilon
\]
\[
\sup_n \mathbb P_n(\|v_n\|_{C^{\gamma}([0,T];H^{-1})}>C_5)\le \varepsilon
\]
\end{proposition}
Actually, $\gamma=\min(\beta,1-\frac d4)$ with $\beta$ and $\delta$ fulfilling
\eqref{cond-zc}. Therefore $0<\gamma<\frac12 $ and $0<\delta<1$.

\subsection{Convergence}
The latter result provides the tightness to pass to the limit. We have
\begin{theorem}\label{th:esist}
Let $v_0 \in H$ and $f \in L^p(0,T;H^{-1})$ for some $p>2$.
If {\bf (G1)}-{\bf (G2)} are satisfied,
then   there exists a martingale solution
$\left((\tilde\Omega,\tilde{\mathbb F}, \tilde{\mathbb P}), \tilde w,\tilde v\right)$ of \eqref{sns};
in addition
\begin{equation}\label{reg-x-uni}
\tilde v \in L^{\frac4{1-g}}(0,T;H^{\frac{1-g}2})
\cap L^{\frac 8d}(0,T;L^4)
 \qquad \tilde{\mathbb P}-a.s.
\end{equation}
Moreover, if
 $d=2$ then   $\tilde v \in C([0,T];H)$ $\tilde{\mathbb P}$-a.s.
\end{theorem}

\begin{proof}
One proceeds as in \cite{BM2013}.

We fix $0<\gamma<\frac 12$ and $0<\delta<1$ appearing in Proposition \ref{stimexvn} and define
the space
\[
Z=L^{\frac 8d}_{\mathrm{w}}(0,T;L^4)\cap C([0,T];U^\prime)
\cap L^2(0,T;H_{\mathrm{loc}})\cap C([0,T];H_{\mathrm{w}})
\]
with the topology
$\mathcal T$ given by  the supremum of the corresponding topologies.
According to Lemma \ref{lemma-tight}, Proposition \ref{stimexvn} provides that
 the sequence of laws of the processes $v_n$ is tight in $Z$.

By the Jakubowski's generalization  of the
Skohorokod Theorem to  nonmetric spaces, see \cite{BM2013} and \cite{Ja1997}, there exist a subsequence
$\{v_{n_k}\}_{k=1}^\infty$,
a stochastic basis $(\tilde \Omega, \tilde{\mathbb F}, \tilde{\mathbb P})$, $Z$-valued
Borel measurable variables $\tilde v$ and $\{\tilde v_k\}_{k=1}^\infty$ such that
for any $k$ the laws of $v_{n_k}$ and $\tilde v_k$ are the same and $\tilde v_k$ converges to $\tilde v$ $\tilde {\mathbb P}$-a.s.
with the topology $\mathcal T$.

Since each $\tilde v_k$ has the same law as $v_{n_k}$, it is a martingale solution to equation \eqref{snsn}; therefore each  process
\[
\tilde M_k(t):=\tilde v_k(t)-\tilde v_k(0)+\int_0^t A \tilde v_k(s)ds +\int_0^t B(\tilde v_k(s),\tilde v_k(s)) ds
-\int_0^t f(s)ds
\]
is a martingale with quadratic variation
\[
\lb\lb\;\tilde M_k\;\rb\rb (t)=\int_0^t G_k(\tilde v_k(s))G_k(\tilde v_k(s))^*ds.
\]
It is now classical to show, see e.g. \cite{BM2013}, that
\[
\langle \tilde M_k(t)-\tilde M(t),\phi\rangle \to 0
\]
for any $\phi\in H^2$ and every $t \in [0,T]$,
where
\[
\tilde M(t)=\tilde v(t)-\tilde v(0)+\int_0^t A \tilde v(s)ds +\int_0^t B(\tilde v(s),\tilde v(s)) ds
-\int_0^t f(s)ds.
\]
There is no convergence of the quadratic variation processes, but the quadratic variation of the more regular process $J^{-g}\tilde M_k$
\[
\lb\lb\;J^{-g}\tilde M_k\rb\rb (t)=\int_0^t  J^{-g}G_k(\tilde v_k(s))G_k(\tilde v_k(s))^* J^{-g}ds.
`\]
is finite and thanks to \eqref{lim-Gn} it converges to
\[
\int_0^t  J^{-g}G(\tilde v(s))G(\tilde v(s))^* J^{-g}ds
\]
as $n \to \infty$.

With usual martingale representation theorem, see e.g. \cite{dpz} (and newer approaches to this result: \cite{Brz+Ondr_2011_AP} and \cite{BGJ_2009}) we can conclude that
there exists a $Y$-cylindrical Wiener process $\tilde w$ such that
\[
J^{-g}\tilde M(t)=\int_0^t J^{-g}G(\tilde v(s)) \,d\tilde w(s).
\]
Therefore $\tilde v$ is a martingale solution to \eqref{sns}.

Finally, \eqref{reg-x-uni} comes from the uniform estimates of Proposition
\ref{stimexvn}.
\end{proof}

\begin{remark}
Our techniques, although primarily devised in order to treat unbounded domains,
can be applied for bounded domains as well.
So, let us compare our assumptions with that of the seminal paper \cite{FG}
by Flandoli and G\c{a}tarek, where only
smooth and bounded domains are considered.

Firstly, let us observe that condition {\bf (G2)},
although involving the  Banach space $H^{-g,4}_\sol$,
is on the same level of regularity as condition {\bf (G1)}. Secondly,
when working in a  box with the periodic boundary conditions
the assumptions {\bf (G1)} and {\bf (G2)} can be reduced to only one assumption
according to Remark \ref{oss-su-G} ii). Indeed,  one can choose  $\{e_j\}$
to be the eigenvalues of the Stokes operator so that the sequence of the
$e_j(x)$ are uniformly bounded in $x$, see, e.g., \cite{Fe97}.\\
On the other hand, in Section 3.4 of \cite{FG}
the authors assume that $d=2$ and
 \begin{eqnarray}\label{eqn-FG-01}
\exists  g \in (0,\tfrac 12) &:&  \mbox{ the map } G:H\to \gamma(Y;H^{- g }) \\
&& \mbox{ is well defined, bounded and continuous.}
\nonumber
\end{eqnarray}
\delc{and that the mapping $H \ni u \mapsto J^{-g}G(u)\in \gamma(Y,H)$
is . }
With these assumptions on the noise, Flandoli and G\c{a}tarek proved, see (25) in \cite{FG} (but written here with
our notations and correcting  a misprint since
the exponent there need to have a $+$ instead of a $-$)
that for some $\eps>0$ and for all $m\ge 1$ there
exists a constant $C$ (dependent on $T$, $m$ and $\eps$ but
independent of $n$) such that
\[
\mathbb E_n \|z_n\|_{L^m(0,T;H^{\frac 12+\eps})}^m
\le C
\]
From this, using the
continuous embedding $H^{\frac 12+\eps}\subset H^{\eps,4}_\sol$ in the two
dimensional case,
they get the mean estimate in the $L^4(0,T;H^{\eps,4}_\sol)$-norm,
which is the basic tool in their proof.
However, one can obtain this latter estimate
as in our Lemma \ref{lemma-z2}, assuming only {\bf (G1)}-{\bf (G2)}
instead of \eqref{eqn-FG-01}. Let us repeat here again that the reason for us being able to weaken the assumptions from \cite{FG} is the use of It\^o integral with values in 2-smooth Banach spaces such as $L^4$.

Finally,
it seems to us that the argument of \cite{FG} to get existence of a martingale
solution would work also for $d=3$ assuming the following stronger version of the previously recalled assumption \eqref{eqn-FG-01},
\begin{eqnarray}\label{eqn-FG-02}
\exists  g \in (0,\tfrac 14) &:&  \mbox{ the map } G:H\to \gamma(Y;H^{- g }) \\
&& \mbox{ is well defined, bounded and continuous.}
\nonumber
\end{eqnarray}
\delc{\[
\exists  g \in (0,\tfrac 14): \; G:H\to \gamma(Y;H^{- g })
\]}
Indeed, for $d=3$ this gives the uniform estimates for the mean value of
$z_n$ in the norm $L^4(0,T;H^{\frac 34+\eps})$.
Since $H^{\frac 34+\eps}\subset H^{\eps,4}_\sol$ when $d=3$,
one gets also the estimates in the norm of
 $L^4(0,T;H^{\eps,4}_\sol)$.
\\
Obviously,  the last assumption \eqref{eqn-FG-02}
is stronger than our condition {\bf (G1)} also for $d=3$.
\end{remark}

\section{Pathwise uniqueness for $d=2$}\label{sec-u}
For the case $d=2$ we will investigate the  pathwise uniqueness.
This means that  given two  processes $v_1,v_2$ solving
\eqref{sns} on the same stochastic basis $(\Omega,\mathbb F,\mathbb P) $
with the same Wiener process $w$,  initial velocity $v_0$
and force $f$, we have
\begin{equation}\label{def-u-path}
\mathbb P\{v_1(t)=v_2(t) \; \text{for all } t \in [0,T]\}=1.
\end{equation}
We are going to prove  the following result.
\begin{theorem}\label{th:uniq}
Assume {\bf (G1)}-{\bf (G2)}-{\bf (G3)}. For $d=2$, if $v_0 \in H$ and $f \in L^p(0,T;H^{-1})$ for some $p>2$,
then
there is pathwise uniqueness for equation \eqref{sns}.
\end{theorem}

\begin{proof}
Set $V=v_1-v_2$; this difference satisfies
\[
dV(t) +[AV(t) +B(v_1(t),v_1(t))-B(v_2(t),v_2(t))]\ dt=[G(v_1(t))-G(v_2(t))]\,dw(t)
\]
with $V(0)=0$; this equation  is equivalent to
\[
dV(t) +[AV(t) +B(V(t),v_1(t))+B(v_2(t),V(t))]\ dt=[G(v_1(t))-G(v_2(t))]\,dw(t).
\]

Following an idea from \cite{BM2015}, we will use the It\^o formula for $d\left(e^{-\int_0^t \psi(s)ds}\|V(t)\|^2_{H^{-g}}\right)$,
by choosing $\psi$ as done in \cite{S}:
\[
\psi(s)=1+L_g^2+2\overline C\left(\|v_1(s)\|_{H^{\frac {1-g}2}}^{\frac4{1-g}}
+\|v_2(s)\|_{H^{\frac {1-g}2}}^{\frac4{1-g}}\right)
\]
with $L_g$ the Lipschitz constant given in {\bf (G3)}
and $\overline C$ the constant appearing later on in \eqref{cxunic}.
We recall that $v_1, v_2 \in L^{\frac4{1-g}}(0,T;H^{\frac{1-g}2})$
$\mathbb P$-a.s., so $\psi\in L^1(0,T)$ $\mathbb P$-a.s.;
moreover, $V \in C([0,T];H)\subset C([0,T];H^{-g})$. \\
We have
\[
d \left(e^{-\int_0^t \psi(s)ds}\|V(t)\|^2_{H^{-g}}\right)=-  \psi(t) e^{-\int_0^t \psi(s)ds}\|V(t)\|^2_{H^{-g}} dt +
e^{-\int_0^t \psi(s)ds}d\|V(t)\|^2_{H^{-g}}
\]
and the latter differential is well defined and given by
\[\begin{split}
\frac 12 d\|V(t)\|^2_{H^{-g}}&=-\|\nabla V(t)\|_{H^{-g}}^2dt
-\langle J^{-g}[B(V(t),v_1(t))+B(v_2(t),V(t))],J^{-g}V(t)\rangle dt
\\
&+\langle J^{-g}[G(v_1(t))-G(v_2(t))]\,dw(t), J^{-g}V(t)\rangle
+\frac 12 \|G(v_1(t))-G(v_2(t))\|_{\gamma(Y;H^{-g})}^2 dt
\end{split}
\]

By means of Lemma \ref{stimaB-1-gamma} and Young inequality we get that
there exists a constant $\overline C$ such that
\begin{equation}\label{cxunic}\begin{split}
|\langle J^{-g}B(V,v_1),J^{-g}V\rangle|
&=\langle J^{-1-g}B(V,v_1),J^{1-g}V\rangle
\\
&\le \|B(V,v_1)\|_{H^{-1-g}}\|V\|_{H^{1-g}}
\\
& \le C  \|V\|^{\frac {1-g}2}_{H^{-g}} \|V\|^{\frac {1+g}2}_{H^{1-g}}
   \|v_1\|_{H^{\frac {1-g}2}} \|V\|_{H^{1-g}}
\\
&=
C  \|V\|^{\frac {1-g}2}_{H^{-g}} \|V\|^{\frac {3+g}2}_{H^{1-g}}
   \|v_1\|_{H^{\frac {1-g}2}}
\\
&\le \frac 14  \|V\|_{H^{1-g}}^2
 +\overline C \|v_1\|_{H^{\frac {1-g}2}}^{\frac4{1-g}} \|V\|^2_{H^{-g}}
 \\
&= \frac 14  \|\nabla V\|_{H^{-g}}^2+\frac 14  \|V\|_{H^{-g}}^2
 +\overline C \|v_1\|_{H^{\frac {1-g}2}}^{\frac4{1-g}} \|V\|^2_{H^{-g}}
\end{split}
\end{equation}
where we used that
$\|V(t)\|^2_{H^{1-g}}=\|\nabla V(t)\|^2_{H^{-g}}+\|V(t)\|^2_{H^{-g}}$
, see \eqref{quadrati}).
Similarly we get the estimate for
$\langle J^{-g}B(v_2,V),J^{-g}V\rangle$.

Therefore, using {\bf (G3)}   we get
\begin{multline}\label{stima:Sch}
d \left(e^{-\int_0^t \psi(s)ds}\|V(t)\|^2_{H^{-g}}\right)
\\\le
e^{-\int_0^t \psi(s)ds}\langle J^{-g}[G(v_1(t))-G(v_2(t))]\,dw(t), J^{-g}V(t)\rangle.
\end{multline}
The r.h.s. is  a local martingale; indeed if we define the stopping time
\[
\tau_N=T\wedge \inf\{t \in [0,T]:\|V(t)\|_{H^{-g}}>N\}
\]
and
\[
M_N(t)=\int_0^{t\wedge \tau_N} e^{-\int_0^r \psi(s)ds}
\langle J^{-g}V(r), J^{-g}[G(v_1(r))-G(v_2(r))]\,dw(r)\rangle
\]
then
\begin{align*}
\mathbb E M_N(t)^2&\le \mathbb E \int_0^{t\wedge \tau_N}
  e^{-2\int_0^r \psi(s)ds}\|V(r)\|_{H^{-g}}^2\|G(v_1(r))-G(v_2(r))\|^2_{\gamma(Y;H^{-g})}dr\\
&\le L_g^2 \mathbb E \int_0^{t\wedge \tau_N}\|V(r)\|_{H^{-g}}^4 dr\\
&\le L_g^2 N^4 t .
\end{align*}
Hence, $M_N$ is a square integable martingale; in particular $\mathbb E M_N(t)=0$ for any $t$.

Therefore, by integrating  \eqref{stima:Sch} over  $[0,t\wedge\tau_N]$
and taking  the expectation we get
\[
\mathbb E e^{-\int_0^{t\wedge \tau_N} \psi(s)ds}\|V(t\wedge \tau_N)\|^2_{H^{-g}}
\le 0.
\]
So
\[
e^{-\int_0^{t\wedge \tau_N} \psi(s)ds}\|V(t\wedge \tau_N)\|^2_{H^{-g}}=0 \qquad \mathbb P-a.s.
\]
Since $\displaystyle\lim_{N\to \infty} \tau_N=T\ \mathbb P$-a.s., we get in the limit that for any $t\in [0,T]$
\[
e^{-\int_0^{t} \psi(s)ds}\|V(t)\|^2_{H^{-g}}=0 \qquad \mathbb P-a.s.
\]
Thus, if we take a  sequence $\{t_k\}_{k=1}^\infty$ which is dense in $[0,T]$ we have
\[
\mathbb P\{\|V(t_k)\|_{H^{-g}}=0 \; \text{for all } k \in \mathbb N\}=1.
\]
Since each path of the process $V$ belongs to $C([0,T]; H^{-g})$, we get \eqref{def-u-path}.
\end{proof}

\medskip
We recall that a strong solution of \eqref{sns} is
a progressively measurable process $v:[0,T]\times \hat \Omega\to H$
fulfilling \eqref{sol-path}
for any given stochastic basis $(\hat\Omega,\hat{\mathbb F},\hat{\mathbb P})$
and  any $Y$-cylindrical Wiener process $\hat w$.

Pathwise uniqueness and existence of martingale solutions imply existence of
stong solution, see, e.g. Ikeda Watanabe monograph \cite{IW}.
Therefore,
Theorems \ref{th:esist} and \ref{th:uniq} imply
\begin{theorem}
Let $d=2$ and $v_0 \in H$, $f \in L^p(0,T;H^{-1})$ for some $p>2$.
If all three assumptions {\bf (G1)}-{\bf (G2)}-{\bf (G3)} are satisfied,
then   there exists a unique strong solution  of \eqref{sns}.
\end{theorem}

\section{Appendix}
\subsection{$\gamma$-radonifying operators}
We refer to \cite{BvN}, \cite{vN} for the definition and main
properties of $\gamma$-radonifying operators.

Let $Y$ be a real separable Hilbert space and $E$ a real separable Banach space;
let $\gamma_Y$ be the standard cylindrical Gaussian measure of $Y$.
We denote by $\mathcal L(Y;E)$ the space of linear bounded operators from
$Y$ to $E$, and by $\gamma(Y;E)$ the space of $\gamma$-radonifying
 operators from $Y$ to $E$. $T\in \gamma(Y;E)$ means that  $T\in\mathcal L(Y;E)$
and  $T(\gamma_Y)$ extends to a Gaussian measure on $E$.

We recall the following well-known facts:
\begin{itemize}
\item  If $T : Y \to  E$ is $\gamma$-radonifying and $S : E \to F$ is
bounded, then also $S \circ T : Y \to F$ is $\gamma$-radonifying;
\item If $T : Y_1 \to E$ is $\gamma$-radonifying and $S : Y_0 \to Y_1$
is bounded, then $T \circ S : Y_0 \to E$ is $\gamma$-radonifying;
\item
If $E$ is a Hilbert space, then $T : Y \to E$ is $\gamma$-radonifying
if and only if $T$ is Hilbert-Schmidt.
\end{itemize}

We have the following characterization of $\gamma$-radonifying operators
when $E=L^p(\mathbb R^d)$, see Proposition 13.7 in \cite{vN} and Theorem 2.3 in  \cite{BvN}.
\begin{proposition}
Let $1\le p < \infty$ and $\{e_j\}_{j=1}^\infty$
a complete orthonormal system in $Y$.
For an operator $T \in \mathcal L(Y,L^p(\mathbb R^d))$ the following
assertions are equivalent:
\begin{enumerate}
\item $T \in \gamma(Y,L^p(\mathbb R^d))$;
\item
$(\sum_{j=1}^\infty |T e_j|^2)^{\frac 12} \in L^p(\mathbb R^d)$.
\end{enumerate}
Moreover the norms $\|T\|_{\gamma(Y,L^p(\mathbb R^d))}$ and
$\|(\sum_j |T e_i|^2)^{\frac 12}\|_{L^p(\mathbb R^d)}$ are equivalent.
\end{proposition}

\bigskip\bigskip
For the stochastic integral
\[
X(t)=\int_0^t \phi(s)dW(s)
\]
of a progressively measurable process
$\phi: [0,T]\times \Omega \to \gamma(Y;E)$
with respect to a $Y$-cylyndrical Wiener process $W$,
we have the following  result, see e.g. \cite{Brz_1997}.
\begin{proposition}\label{conv-stoc}
Let  $W$ be a $Y$-cylindrical Wiener process. \\
If for some $m \ge 2$ we have $\mathbb E[(\int_0^T \|\phi(t)\|_{\gamma(Y;E)}^2dt)^{m/2}]<\infty$,
then $X$ has a
 progressively measurable $E$-valued version and
\[
\mathbb E \|X(t)\|_E^m\le C_m \mathbb E \left[
\left(\int_0^t \|\phi(s)\|_{\gamma(Y;E)}^2ds\right)^{\frac m2} \right].
\]
\end{proposition}

\subsection{Compactness lemmas}
In \cite{MR2005,BM2013} there are some useful compactness results.
We follow \cite{BM2013}.

Given $1<p<\infty$, we denote by $L^p_{\mathrm{w}}(0,T;L^4)$ the space $L^p(0,T;L^4)$
with the weak topology.

For any $R>0$ let $B_R=\{x \in \mathbb R^d:|x|<R\}$. Given
$v: \mathbb R^d\to \mathbb R^d$,
we denote by
$v|_{B_R}$ its restriction to the ball  $B_R$.
We define
\[
H_R=\{v|_{B_R}: v \in H\},\qquad
H_R^\delta=\{v|_{B_R}: v \in H^\delta\}
\]
with
\[
\|v\|_{H_R}=\left( \int_{B_R} |v(x)|^2 dx \right)^{1/2},
\qquad
\|v\|_{H^\delta_R}=\left( \int_{B_R} |(\Lambda^\delta v)(x)|^2 dx \right)^{1/2}.
\]
For $\delta>0$ the space $H^\delta_R$ is compactly embedded in $H_R$, since
the space variable belongs to a bounded set.

We denote by $L^2(0,T;H_{\mathrm{loc}})$ the space of measurable functions
$v:[0,T]\to H$ such that  for any $R>0$ the norm
\[
\|v\|_{L^2(0,T;H_R)}=\left(\int_0^T \int_{B_R} |v(t,x)|^2 dx \ dt \right)^{1/2}
\]
is finite.
It is a Fr\'echet space with the topology
generated by the seminorms $\|v\|_{L^2(0,T;H_R)}$, $R\in \mathbb N$.

\begin{lemma}
Let
\[
\tilde Z=L^{p}_{\mathrm{w}}(0,T;L^4)\cap C([0,T];U^\prime) \cap L^2(0,T;H_{\mathrm{loc}})
\]
for some $p \in (1,\infty)$,
and let $\tilde {\mathcal T}$ be the supremum of the corresponding topologies.
\\
Then a set $\tilde K\subset \tilde Z$ is $\tilde {\mathcal T}$-relatively
compact if the following conditions hold:

(i) $\displaystyle\sup_{v \in \tilde K} \|v\|_{L^{p}(0,T;L^4)}<\infty$

(ii) $\displaystyle\exists \gamma>0: \sup_{v \in \tilde K}\|v\|_{C^\gamma([0,T];H^{-1})}<\infty$

(iii) $\displaystyle\exists \delta>0: \sup_{v \in \tilde K} \|v\|_{L^2(0,T;H^\delta)}<\infty$.
\end{lemma}

\begin{proof}
We can assume that $\tilde K$ is closed in $\tilde {\mathcal T}$.

It is trivial from (i) that the
set $\tilde K$ is compact in $L^{p}_{\mathrm{w}}(0,T;L^4)$; this comes from the
Banach-Alaoglu theorem.
Moreover (ii) implies that
the functions $v\in\tilde K$ are equicontinuous, i.e.
\[
\forall \varepsilon>0 \; \exists \delta >0 :
 |t-s|<\delta \Longrightarrow
\|v(t)-v(s)\|_{H^{-1}} \le \varepsilon \qquad
\forall v \in \tilde K.
\]
Since $H^{-1}$ is compactly embedded in $U^\prime$, Ascoli-Arzel\`a theorem
provides that $\tilde K$ is compactly embedded in  $C([0,T];U^\prime)$.

Notice that the compactness of a subset of $\tilde Z$ is equivalent
to its sequential compactness. Therefore,  if we take a sequence
$\{v_i\}_{i\in \mathbb N}\subset \tilde K$ then there exists a subsequence
converging to some $v \in  \tilde K$ in the two previous topologies.
What remains to prove is
that this subsequence  (or possibly
 a subsubsequence) is convergent in $L^2(0,T;H_{\mathrm{loc}})$.
\\
Let us fix $R>0$; the embedding $H^\delta_R\subset H_R$ is compact and
$H_R\simeq H^\prime_R \subset H^\prime\subset U^\prime$ with  continuous
embeddings. Hence by a Lions
Lemma, see \cite{l},
for every $\varepsilon >0$ there exists a constant $C=C_{\varepsilon,R}$
such that
\[
\|v\|^2_{H_R}\le \varepsilon \|v\|^2_{H^\delta_R}+C\|v\|^2_{U^\prime}.
\]
Thus for almost all $s \in [0,T]$
\[
\|v_n(s)-v(s)\|^2_{H_R}\le \varepsilon \|v_n(s)-v(s)\|^2_{H^\delta_R}
   +C\|v_n(s)-v(s)\|^2_{U^\prime}
\]
and therefore, setting $M=\sup_{v \in K}  \|v\|^2_{L^2(0,T;H^\delta)}$ we get
\begin{align*}
\|v_n-v\|^2_{L^2(0,T;H_R)}&\le \varepsilon \|v_n-v\|^2_{L^2(0,T;H^\delta_R)}
   +C\|v_n-v\|^2_{L^2(0,T;U^\prime)}\\
&\le
\varepsilon 2M
   +C T \|v_n-v\|^2_{C([0,T];U^\prime)}
\end{align*}
Now we pass to the upper limit (in a subsequence)  and get
\[
\limsup_{j \to \infty} \|v_{n_j}-v\|^2_{L^2(0,T;H_R)}\le 2M \varepsilon ;
\]
since $\varepsilon$ is arbitrary,  we get the convergence in $L^2(0,T;H_{\mathrm{loc}})$.

\end{proof}

If we proceed as in Lemma 3.3 in \cite{BM2013} or Lemma 2.7 in \cite{MR2005},
we get also the following result.
\begin{lemma}
Let
\[
Z=L^{p}_{\mathrm{w}}(0,T;L^4)\cap C([0,T];U^\prime) \cap L^2(0,T;H_{\mathrm{loc}})\cap C([0,T];H_{\mathrm{w}})
\]
for some $p \in (1,\infty)$,
and let ${\mathcal T}$ be the supremum of the corresponding topologies.
\\
Then a set $K\subset Z$ is $\mathcal T$-relatively
compact if the following conditions hold:

(i) $\displaystyle\sup_{v \in K} \|v\|_{L^{p}(0,T;L^4)}<\infty$

(ii) $\displaystyle\exists \gamma>0: \sup_{v \in K}\|v\|_{C^\gamma([0,T];H^{-1})}<\infty$

(iii) $\displaystyle\exists \delta>0: \sup_{v \in K} \|v\|_{L^2(0,T;H^\delta)}<\infty$

(iv) $\displaystyle \sup_{v \in K} \|v\|_{L^\infty(0,T;H)}<\infty$.
\end{lemma}

From this lemma we also get a tightness criterion.
We recall that a family of probability measures $\{P_n\}_n$, defined
on the $\sigma$-algebra of Borel subsets of $Z$, is tight
if for any $\varepsilon>0$ there exists a
compact subset $K_\varepsilon$ of $Z$ such that
\[
\inf_n P_n(K_\varepsilon)\ge 1-\varepsilon
\]
or equivalently
\[
\sup_n P_n(Z\setminus K_\varepsilon)\le \varepsilon.
\]

\begin{lemma}[tightness criterion] \label{lemma-tight}
We are given parameters $\gamma >0$, $\delta>0$, $1<p<\infty$ and
a sequence $\{v_n\}_{n \in \mathbb N}$ of adapted processes in $C([0,T];U^\prime)$.

Assume that for any $\varepsilon>0$ there exist positive constants
$R_i=R_i(\varepsilon)$ ($i = 1, \ldots, 4$)
such that
\begin{align}
&\displaystyle\sup_n \mathbb{P}(\|v_n\|_{L^{p}(0,T;L^4)}>R_1) \le \varepsilon\\
&\displaystyle\sup_n \mathbb{P}(\|v_n\|_{C^\gamma([0,T];H^{-1})}>R_2) \le \varepsilon\\
&\displaystyle\sup_n \mathbb{P}(\|v_n\|_{L^2(0,T;H^\delta)}>R_3) \le \varepsilon\\
&\displaystyle\sup_n \mathbb{P}(\|v_n\|_{L^\infty(0,T;H)}>R_4) \le \varepsilon
\end{align}
Let $\mu_n$ be the law of $v_n$ on
$Z=L^{p}_{\mathrm{w}}(0,T;L^4)\cap C([0,T];U^\prime) \cap L^2(0,T;H_{\mathrm{loc}})\cap C([0,T];H_{\mathrm{w}})$.
Then, the sequence $\{\mu_n\}_{n\in \mathbb N}$ is tight in $Z$.
\end{lemma}

\medskip
\noindent
{\bf Acknowledgements}.
Part of this research started while Z. Brze{\'z}niak   was visiting the
Department of Mathematics of the University of Pavia and was partially
supported by  the GNAMPA-INDAM project "Regolarit\`a e dissipazione in
fluidodinamica" and the PRIN 2010-2011;
he would like to thank the hospitality of the Department.

\end{document}